\documentclass[sn-mathphys-num]{sn-jnl}

\usepackage{graphicx}%
\usepackage{multirow}%
\usepackage{amsmath,amssymb,amsfonts}%
\usepackage{amsthm}%
\usepackage{mathrsfs}%
\usepackage[title]{appendix}%
\usepackage{xcolor}%
\usepackage{textcomp}%
\usepackage{manyfoot}%
\usepackage{booktabs}%
\usepackage{algorithm}%
\usepackage{algorithmicx}%
\usepackage{algpseudocode}%
\usepackage{listings}%
\usepackage{bbm}

\theoremstyle{thmstyleone}%
\newtheorem{theorem}{Theorem}
\newtheorem{proposition}[theorem]{Proposition}%

\newtheorem{lemma}[theorem]{Lemma}

\theoremstyle{thmstyletwo}%

\theoremstyle{thmstylethree}%

\newcommand{\N}{{\mathbb N}}

\newcommand{\blue}[1]{{\color{black}#1}}

\begin{document}

\title[Condensation in ZRP with fast rate]{Condensation in zero-range processes with a fast rate}


\author[1]{\fnm{Watthanan} \sur{Jatuviriyapornchai}}\email{watthanan.jat@mahidol.ac.th}

\author*[2]{\fnm{Stefan} \sur{Grosskinsky}}\email{Stefan.Grosskinsky@math.uni-augsburg.de}


\affil[1]{\small\orgdiv{Department of Mathematics, Faculty of Science}, \orgname{Mahidol University}, \orgaddress{\street{Rama VI Road}, \city{Ratchathewi}, \state{Bangkok}, \postcode{10400}, \country{Thailand}}}

\affil*[2]{\small\orgdiv{Lehrstuhl f\"ur Stochastik und ihre Anwendungen}, \orgname{ Universit\"at Augsburg}, \orgaddress{\street{Universit\"atsstr. 14}, \city{Augsburg}, \postcode{86159}, \country{Germany}}}



\abstract{We introduce a simple zero-range process with constant rates and one fast rate for a particular occupation number, which diverges with the system size. Surprisingly, this minor modification induces a condensation transition in the thermodynamic limit, where the structure of the condensed phase depends on the scaling of the fast rate. We study this transition and its dependence on system parameters in detail on a rigorous level using size-biased sampling. This approach generalizes to any particle system with stationary product measures, and the techniques used in this paper provide a foundation for a more systematic understanding of condensing models with a non-trivial condensed phase.}




\maketitle

\section{Introduction}

Condensation in interacting particle systems has been a topic of continued research interest in recent years. 
When the particle density exceeds a critical value, the system phase separates into a homogeneous background or bulk phase and a condensate, where a non-zero fraction of the mass concentrates on a vanishing volume fraction. Building on the first few results for zero-range processes \cite{o1998jamming,evans2000phase,jeon2000size,grosskinsky2003condensation,godreche2003dynamics}, condensation has been observed in various models and scaling limits, including inclusion processes \cite{grosskinsky2011condensation} and non-linear extensions thereof \cite{evans2014condensation}. In addition to stationary results, \blue{the formation of the condensate from homogeneous initial conditions has attracted attention \cite{beltran2017martingale,godreche2016coarsening,jatuviriyapornchai2016coarsening,armendariz2023fluid}} as well as stationary dynamics in the context of metastability. In all of these examples, the condensate consists asymptotically only of a single cluster site, the location of which exhibits metastable dynamics on a slow time scale. This has been established for zero-range \cite{beltran2012metastability,armendariz2015metastability,seo2019condensation,choi2024} and inclusion processes \cite{grosskinsky2013dynamics,bianchi2017metastability,kim2021condensation}, see also \cite{beltran2015martingale,landim2019metastable} and references therein. For the inclusion process, condensation only occurs when the rate of independent particle diffusion vanishes with the system size \cite{grosskinsky2011condensation}. Models with size-dependent parameters have also been studied for zero-range processes \cite{grosskinsky2008discontinuous,chleboun2015dynamical}. \blue{Another important topic has been the role of spatial inhomogeneities (see \cite{godreche2012condensation} and references therein), which we do not consider in this paper.}

A particular question of recurring interest has been to identify spatially homogeneous models where the condensed phase exhibits a more interesting structure. This has for example been observed in variants of the zero-range process with cut-off \cite{schwarzkopf2008zero} or systems with pair-factorized stationary measures \cite{evans2006interaction,evans2015condensation}. More recently, the condensed phase of the inclusion processes with a moderately small diffusion rate has been shown to consist of a diverging number of independent cluster sites or exhibit an interesting hierarchical structure in a particular scaling \cite{jatuviriyapornchai2020structure}, given by the Poisson-Dirichlet distribution of mass partitions. The latter has been observed in various contexts including population genetic models and random partitions, see \cite{feng2010poisson} for a general overview. The dynamics of the inclusion process in these scaling limits on complete graphs has recently been studied in \cite{chleboun2023size} and extensions of stationary Poisson-Dirichlet statistics to other spatially homogeneous models in \cite{chleboun2022poisson}.

The aim of this paper is to study a new example of a condensing particle system that exhibits a non-trivial condensate, which is interesting primarily due to its simplicity. Consider a spatially homogeneous zero-range process with a constant jump rate equal to $1$, except if the occupation number of a site takes a particular value $A>1$. A site that contains $A$ particles loses one of them with a large rate $\theta_L \gg 1$ that diverges with the system size. This introduces a weak form of an exclusion interaction, where occupation numbers or cluster sizes of size $A$ are unstable and all other occupation numbers are stable. It is surprising that these dynamics do not only diminish the stationary probability of observing occupation numbers of size $A$ but lead to a condensation transition in the system, where the condensate exhibits clusters on a scale depending on the scaling of $\theta_L$. \blue{In fact, we can keep the jump rates for occupation numbers smaller than $A$ general, which only affect the bulk distribution.} The structure of the condensate depends crucially on jump rates for occupation numbers larger than $A$, which we discuss in detail for several examples. 

The paper is structured as follows: In Section \ref{sec:setting} we introduce the model and its stationary distributions and give a first heuristic argument for the condensation phenomenon. In Section \ref{sec:results} we formulate rigorous results on the condensation transition in the thermodynamic scaling limit and present the proofs based on size-biased sampling. For completeness, we also include results on a scaling limit with fixed volume and diverging density. We discuss the robustness of our results and the relevance of the size-biased approach for more general systems in Section \ref{sec:conclusion}.

\section{Mathematical Setting\label{sec:setting}}

\subsection{The zero-range process and its stationary distributions}

We consider a zero-range process (ZRP) on a finite set of sites $\Lambda$ with size $L=|\Lambda |$, called lattice in the following, with an infinitesimal generator of the form
\begin{equation}\label{gene}
\blue{\mathcal{L} f(\eta) = 
\sum_{x,y \in \Lambda} p(x,y) g(\eta_x ) \left[
f(\eta^{xy}) - f(\eta)
\right]}\, .
\end{equation}
\blue{Here $\eta \in E_{L, N} =\big\{\eta \in\N_0^L :\sum_{x\in\Lambda} \eta_x =N\big\}$, where $\eta_x \in\N_0$ denotes the number of particles on site $x$ and $f\in C_b (E_{L,N})$ is a continuous test function.}
Single particles jump from site $x$ to $y$, changing the configuration from $\eta$ to $\eta^{xy}$ with $\eta^{xy}_z =\eta_z -\delta_{z,x} +\delta_{z,y}$ \blue{provided that $\eta_z>0$}. This jump occurs with rate $p(x,y)g(\eta_x )$, where we assume that the spatial factor is irreducible and doubly stochastic, i.e.,
\[
\sum_{y\in\Lambda} \big( p(x,y)-p(y,x)\big) =0\quad\mbox{for all }x\in\Lambda\ .
\]
The rate $g:\N_0 \to [0,\infty )$ fulfills the usual irreducibility condition $g(n)=0$ if and only if $n=0$. Under these conditions, the generator \blue{$\mathcal{L}$} defines an irreducible continuous-time Markov chain on the finite state space $E_{L, N}$, and it is well known \cite{andjel1982invariant} that its unique stationary measure is spatially homogeneous and of product form
\begin{equation}\label{wpi}
\pi_{L,N}[d\eta] =
\dfrac{1}{Z_{L,N}}\prod_{x\in\Lambda} w (\eta_x)d\eta\ ,\quad\eta\in E_{L,N}\ ,
\end{equation}
with stationary weights
\begin{equation}\label{eq: weights}
w(n)=\prod_{k=1}^n \frac{1}{g(k)} ,\ n\geq 0\quad\mbox{and normalization}\quad Z_{L,N} =\sum_{\eta\in E_{L,N}} \prod_{x\in\Lambda} w(\eta_x )\ .    
\end{equation}
We will be interested in a ZRP with size-dependent rates of the form
\begin{equation}
\label{eq_rates}
g_L(n)  =
\begin{cases}
0 &\quad \text{ if } n=0\, ,\\
g(n)>0 &\quad\text{ if } n=1,\ldots, A-1\, \blue{,}\\
\theta_L &\quad \text{ if } n=A\, ,\\
1 &\quad \text{ if }n>A\, ,
\end{cases}
\end{equation}
where $A\geq 2$ and $\theta_L \to\infty$ for $L\to\infty$. This leads to stationary weights
\begin{equation}
\label{eq_weight}
w_L(n)  = \prod_{k=1}^n \frac{1}{g_L (k)} =
\begin{cases}
w(n) &\quad \text{ if } n< A\, ,\\
\frac{w(A-1)}{\theta_L}=\frac{1}{\theta_L} &\quad \text{ if } n \geq A\, ,
\end{cases}
\end{equation}
with $w(n)$, $n\leq A-1$ as in \eqref{eq: weights}. For simplicity of notation, we want to absorb the factor $w(A-1)$ for $n\geq A$ in the definition of $\theta_L$ in the following.

\blue{While all results in the paper apply for general $g(n)$ and $w(n)$ as defined in \eqref{eq_rates} and \eqref{eq_weight}, as a particularly simple example we can think of $g_L (n)=1$ for $n\neq 0,A$ and $g(A)=\theta_L$, so that $w_L (n)=w(n)=1$ for $n<A$ and $w_L (n)=\theta_L^{-1}$ for $n\geq A$.}\\

\subsection{Grand-canonical distributions and a first heuristic\label{sec:heuristics}}

In addition to the canonical measures $\pi_{L,N}$ with a fixed number of particles, the ZRP also has a family of grand-canonical product measures \cite{andjel1982invariant,schuetz2008}
\begin{equation}\label{nu}
\nu_\phi^L [d\eta] =
\prod_{x\in\Lambda} \frac{1}{z_L (\phi )}w_L (\eta_x) \phi^{\eta_x} d\eta\ ,\quad\eta\in\N_0^L\ ,
\end{equation}
with single-site normalization
\begin{equation}\label{zlphi}
z_L(\phi) 
= \sum_{n\geq 0} \phi^nw_L(n) 
= \sum_{n=0}^{A-1} \phi^n w(n) +\frac{1}{\theta_L}\left(\frac{\phi^A}{1-\phi}\right)\ .
\end{equation}
The fugacity parameter $\phi\geq 0$ regulates the expectation of the (random) number of particles in the system and the measures \eqref{nu} are defined for all $\phi\in [0,1)$. The particle density per site is given by
\begin{align}
    R_L (\phi )&=\sum_{n\geq 1} n\frac{w_L (n)\phi^n}{z_L (\phi )} =\phi\frac{z'_L (\phi )}{z_L (\phi )}\nonumber\\
    &=\frac{1}{z_L (\phi )}\bigg( \sum_{n=0}^{A-1} n\phi^n w(n) +\frac{1}{\theta_L}\Big(\frac{\phi^A (A(1-\phi )+\phi )}{(1-\phi )^2}\Big)\bigg) \label{rhol}
\end{align}
and for $L\to\infty$ we have for all $\phi\in [0,1 )$
\begin{equation}\label{rzlim}
    z_L (\phi )\to z(\phi ):=\sum_{n=0}^{A-1} \phi^n w(n)\quad\mbox{and}\quad R_L (\phi )\to R(\phi ):=\frac{1}{z(\phi )}\sum_{n=1}^{A-1} n\phi^n w(n)\ .
\end{equation}
The limiting expressions can be extended to $\phi \nearrow 1$ and we get
\begin{equation}\label{extend}
    z(\phi )\to z(1)=\sum_{n=0}^{A-1} w(n)<\infty\ ,\quad R(\phi )\to R(1)=\rho_c :=\frac{1}{z(1)}\sum_{n=1}^{A-1} n w(n)<\infty\ .
\end{equation}
\begin{figure}
    \begin{center}
\mbox{\includegraphics[width=0.5\textwidth]{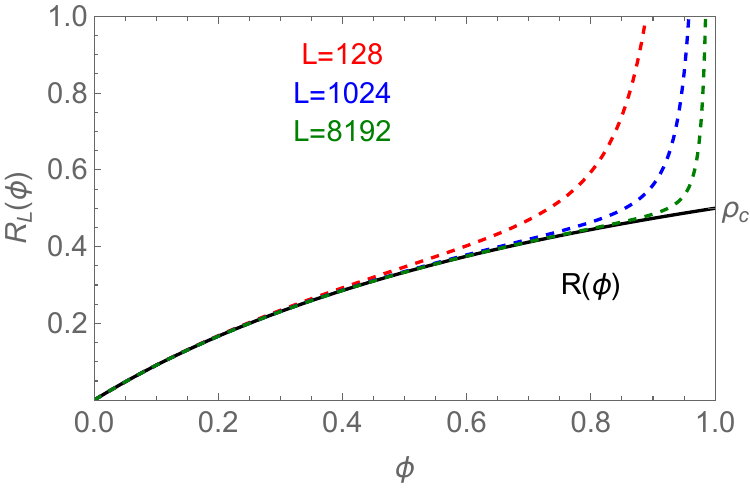}\quad\includegraphics[width=0.5\textwidth]{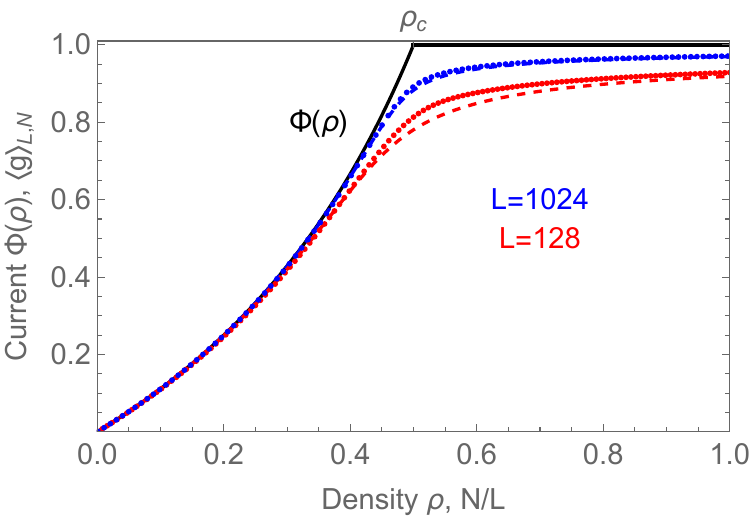}
    }
    \end{center}
    \caption{\label{fig:gc}
    The density $R_L (\phi )$ \eqref{rhol} (dashed coloured lines) converges pointwise for all $\phi <1$ to $R(\phi )$ (black line) for $L\to\infty$ \eqref{rzlim} (left). The canonical current \eqref{cancurr} (dotted lines) compares well with the grand-canonical prediction $\Phi_L (\rho )$ \eqref{gcancurr} (dashed lines). For $L\to\infty$ both converge to the same limit $\Phi (\rho )$ (black line), which is the inverse of $R(\phi )$ \eqref{rzlim}. Parameters are $A=2$, $w(n)=1$ for $n=0,1$ and $\theta_L =1/L$.
    }
\end{figure}

\noindent This convergence is illustrated in Figure \ref{fig:gc}. Note that the limits $L\to\infty$ and $\phi\nearrow 1$ do not commute since $z_L (\phi )$ and $R_L (\phi )$ diverge for $\phi\nearrow 1$. In analogy to zero-range models with size-dependent parameters studied in \cite{schuetz2008}, this corresponds to the following condensation transition. In the thermodynamic limit $L,N\to\infty$ with $N/L\to\rho\geq 0$, when the particle density $\rho$ exceeds the critical density $\rho_c$, the system separates into a homogeneous bulk or background phase with density $\rho_c$ and the excess mass $N-\rho_c L$ concentrates on a vanishing volume fraction (the condensate or condensed phase). While the condensate in most systems studied so far consists only of a single site, it exhibits an interesting non-trivial structure in the present model depending on the scaling of the parameter $\theta_L$. Before this is established rigorously in the next section, we give a first heuristic on the scale of cluster sizes in the condensed phase based on the grand-canonical measures.


We see in Figure \ref{fig:gc} (left) that for fixed $L$ we can find $\Phi_L(\rho )=1-\Delta_L <1$ such that $R_L (1-\Delta_L )=\rho$ for any $\rho\geq 0$, including $\rho >\rho_c$. With \eqref{rhol} this leads to
\[
R_L (1-\Delta_L )=\rho_c +O(\Delta_L )+\frac{(1-A\Delta_L )((A-1)\Delta_L +1)}{\theta_L z_L (1-\Delta_L )\Delta_L^2} = \rho  \ ,
\]
and with $z_L (1-\Delta_L )=z(1)+O(\Delta_L) +O(1/(\theta_L \Delta_L))$ we get to leading order
\[
\rho -\rho_c =\frac{1}{z(1)\theta_L \Delta_L^2}\big(1+O(\Delta_L )\big)\ .
\]
This implies that for $\rho >\rho_c$ to leading order
\begin{equation}\label{gcancurr}
    \Delta_L \simeq\frac{1}{\sqrt{(\rho -\rho_c)z(1)\theta_L}}\quad\mbox{so that}\quad \Phi_L (\rho )\simeq 1-\frac{1}{\sqrt{(\rho -\rho_c)z(1)\theta_L}}\ ,
\end{equation}
where we use the notation $a_L \simeq b_L$ for sequences with $a_L /b_L \to 1$, $L\to\infty$. 
Together with \eqref{nu} and \eqref{zlphi} we get to leading order
\begin{equation}\label{geo}
\nu^L_{\Phi_L (\rho )} [\eta_x =k|\eta_x \geq A] =\frac{(1-\Delta_L )^{k-A}}{\theta_L}\Big/\frac{1}{\theta_L \Delta_L} =\Delta_L (1-\Delta_L )^{k-A} \ ,\quad k\geq A\ ,    
\end{equation}
i.e. cluster sites with occupation numbers $\geq A$ are geometrically distributed with expectation
\begin{equation}\label{scale}
C_L :=1/\Delta_L =\sqrt{(\rho -\rho_c )z(1)\theta_L}\quad\mbox{(scale of cluster sizes)}\ .
\end{equation}
Scaling $k=k_L$ such that $k_L /C_L =k_L \Delta_L \to u\geq 0$ we get an exponential distribution for the asymptotic cluster size on scale $C_L$, i.e.
\begin{equation}\label{heuri}
\nu^L_{\Phi_L (\rho )} [\eta_x \geq k_L |\eta_x \geq A] \simeq \Big( 1-\frac{1}{C_L} \Big)^{uC_L} \to e^{-u}\quad\mbox{as }L\to\infty\ .
\end{equation}
The geometric distribution in \eqref{geo} is consistent with effective birth-death dynamics of cluster sites with exit rate $1$ and entrance rate $\Phi_L (\rho )$. The latter is well known (see e.g. \cite{godreche2003dynamics,grosskinsky2003condensation}) to be the stationary current or activity in the ZRP in the grand-canonical ensemble at density $\rho$, and is illustrated as the inverse of $R_L (\phi )$ in Figure \ref{fig:gc} (right). The grand-canonical current $\Phi_L (\rho )$ is compared with the canonical current under the distribution $\pi_{L,N}$. Using \eqref{eq: weights} this is in general given by
\begin{equation}\label{cancurr}
\langle g\rangle_{L,N} :=\sum_{n=1}^N g(n)\pi_{L,N} [\eta_x =n]=\frac{Z_{L,N-1}}{Z_{L,N}} \ ,
\end{equation}
which can be computed recursively via $Z_{L,N} =\sum_{n=0}^N Z_{L/2 ,n} Z_{L/2 ,N-n}$ (cf.\ e.g.\ \cite{chleboun2010finite}).\\

\begin{figure}
    \begin{center}
\mbox{\includegraphics[width=0.5\textwidth]{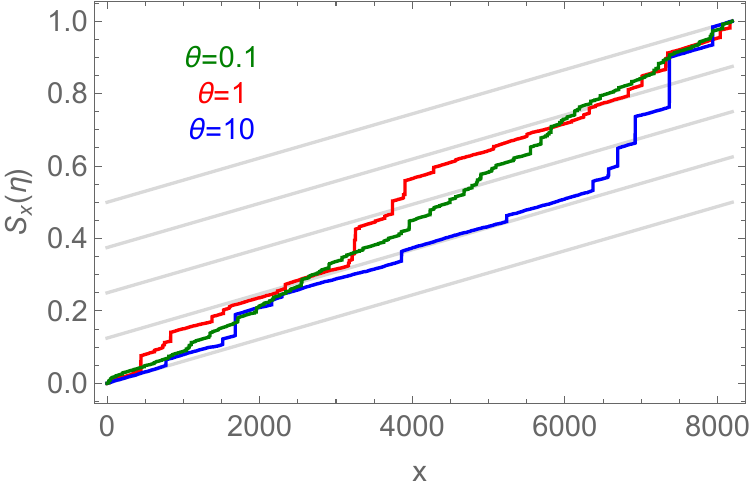}\quad\includegraphics[width=0.48\textwidth]{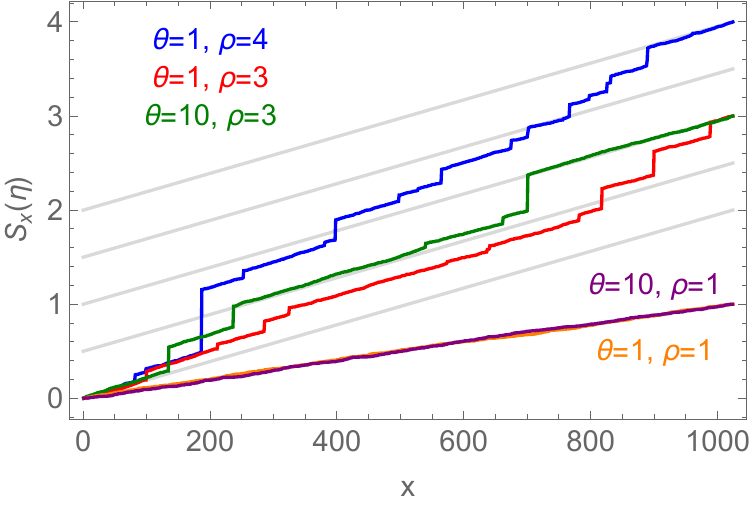}
    }
    \end{center}
    \caption{\label{fig:sim}
Integrated stationary density profiles $S_x (\eta )=\sum_{k=1}^x \eta_x$ for a system with parameters \eqref{simple} and several values for $\Theta =\theta_L /L$ confirm condensation with bulk density $\rho_c$ \eqref{extend} (indicated by grey lines). We have $A=2$, $L=8192$, $\rho =1$ with $\rho_c =1/2$ (left) and $A=5$, $L=1024$, $\rho =3$ and $4$ with $\rho_c =2$, \blue{as well as subcritical $\rho =1$} (right). The size of clusters is increasing with $\Theta$ in accordance with the scale $C_L$ \eqref{scale}.
}\end{figure}

This heuristic is confirmed in Figure \ref{fig:sim}. We simulate the process to stationarity and show integrated density profiles from single realizations for different parameter values of $A$ and $\Theta=\lim_{L\to\infty} \theta_L /L$ for the simple ZRP with
\begin{equation}\label{simple}
    g_L(n)  =
\begin{cases}
1 &,\ n=1,\blue{\dots}, A{-}1\, \\
\theta_L &,\ n=A\, \\
1 &,\ n>A \, 
\end{cases}\ ,\quad w_L(n)  =
\begin{cases}
1 &,\ n<A\\
1/\theta_L &,\ n\geq A
\end{cases}\ ,\quad \rho_c =\frac{A-1}{2} \ .
\end{equation}
Here $z(1)=A$ and we have a uniform distribution in the bulk. 

The equivalence of ensembles between grand-canonical and canonical distributions has been established in the bulk phase for many condensing particle systems \cite{grosskinsky2003condensation,grosskinsky2008equivalence,grosskinsky2008discontinuous,evans2014condensation,chleboun2014condensation,huveneers2019equivalence}. In the next section, we will show that this extends also to the condensed phase in our model so that the above prediction on the asymptotic cluster size distribution is correct, and that furthermore, cluster sites are also asymptotically independent.\\




\section{\blue{Rigorous Results on Condensation}\label{sec:results}}

\subsection{\blue{Main results}}
In this section, we give a precise version of the above heuristics of the condensation transition in the thermodynamic limit
\begin{equation}\label{thermo}
    L,N,\theta_L \to\infty\quad\mbox{such that}\quad \frac{N}{L}\to\rho >0\ ,\quad \frac{\theta_L}{L^\gamma}\to\Theta >0\mbox{ with }\gamma >0\ .
\end{equation}
Our results describe the limiting stationary behaviour under the canonical distributions $\pi_{L,N}$ \eqref{wpi}, which are uniquely determined by the weights \eqref{eq_weight} so that the explicit dynamics of the ZRP can be ignored from now on.\\

\begin{theorem}\label{thm1}
    The ZRP with stationary weights \eqref{eq_weight} exhibits a condensation transition in the thermodynamic limit \eqref{thermo} with critical density \eqref{extend}
    \begin{equation}\label{rhoc}
        \rho_c =\sum_{n=1}^{A-1} n w(n)/\sum_{n=1}^{A-1} w(n)\ .
    \end{equation}
    This means that for any distinct (fixed) lattice sites $x_1 ,\ldots ,x_m \in\Lambda$ (for $L$ large enough) $\eta_{x_1} ,\ldots ,\eta_{x_m}$ converge in distribution under $\pi_{L,N}$ to iid random variables on $\{ 0,\ldots ,A-1\}$ with marginal distribution (see \eqref{rzlim})
    \begin{align*}
    \nu_\phi [\eta_{x_i} =k] &=\frac{w(k)\phi^k}{z(\phi )}\qquad\qquad\qquad\ ,\quad\;\mbox{for }R(\phi )=\rho\leq\rho_c \, ,\\
    \nu_1 [\eta_{x_i} =k] &=\frac{w(k)}{w(0)+ \ldots +w(A-1)}\ ,\quad\mbox{for }\rho\geq\rho_c \ .   
    \end{align*}
\end{theorem}

\noindent The proof is given in the next subsection, and this is a direct consequence of a classical equivalence of ensembles resulting from condensing particle systems. So in any finite test volume, we will only observe the bulk phase in the limit. In order to capture the distribution of the condensed phase it is convenient to work with a size-biased reordering $\tilde\eta$ of the particle configurations $\eta$ (see e.g. \cite{jatuviriyapornchai2020structure,chleboun2022poisson} for more details on this). For a given configuration $\eta\in E_{L,N}$, we pick the first site randomly according to its mass fraction, i.e.
\[
\tilde\eta_1 =\eta_x \quad\mbox{with probability }\frac{\eta_x}{N},\ x\in\Lambda\ .
\]
\blue{Alternatively, we can think of picking a particle uniformly at random, which has location $x$ with probability $\eta_x /N$.} 
Keeping track of the chosen site indices $x_1 ,\ldots x_k$ we repeat recursively as
\[
\tilde\eta_{k+1} =\eta_x \quad\mbox{with probability }\frac{\eta_x}{N-(\eta_{x_1} +\ldots +\eta_{x_k})},\ x\in\Lambda\setminus\{ x_1 ,\ldots ,x_k \}\ .
\]
to produce a size-biased (random) reordering $\tilde\eta$ of $\eta$ where all empty sites are stacked at the end. This has the advantage, that cluster sites have a non-zero probability to be observed in $\tilde\eta$, since the total mass fraction in the condensate is $(\rho -\rho_c )/\rho >0$ while its volume fraction vanishes. Of course, any spatial information from $\eta$ is lost in $\tilde\eta$, but since we only consider spatially homogeneous systems with stationary product measures \eqref{wpi} this is not a restriction. 
\blue{If a condensed configuration consists for instance of a single condensate site and a homogeneous background at density $\rho_c$, then $\tilde\eta_1$ corresponds to the condensate with asymptotic probability $(\rho -\rho_c )/\rho$ or to a bulk site with probability $\rho_c /\rho$. Both are of order $1$, since bulk and condensate both have a non-zero fraction of the total mass, as opposed to the vanishing volume fraction of the condensate.}

The marginal distributions of size-biased configurations are given in terms of size-biased weights \eqref{wpi}
\begin{equation}\label{sb}
\pi_{L,N} [\tilde\eta_1 =n] =\frac{L}{N} n\, w_L (n)\,\frac{Z_{L-1,N-n}}{Z_{L,N}} \ ,\quad n\in \{ 1,\ldots ,N\}
\end{equation}
which is normalized since the particle density under $\pi_{L,N}$ is fixed to $N/L$. Due to the product structure of $\pi_{L,N}$ we get recursively from the definition of $\tilde\eta$
\begin{align}
\pi_{L,N} [\tilde\eta_1 =n_1 ,\ldots ,\tilde\eta_m =n_m ] =&\frac{L(L-1)\cdots (L-m+1)}{N(N-n_1)\cdots (N-(n_1 +\ldots+n_m))}\nonumber\\
& n_1 w_L (n_1)\cdots n_m w(n_m )\,\frac{Z_{L-m,N-(n_1 +\ldots +n_m)}}{Z_{L,N}}\ ,\label{pf}
\end{align}
for any fixed $m\in\N$, $n_i \geq 1$ so that $n_1 +\ldots +n_m \leq N$. This can also be written as
\begin{align}
\pi_{L,N} [\tilde\eta_1 =n_1 ,\ldots ,\tilde\eta_m =n_m ] =&\pi_{L,N} [\tilde\eta_1 =n_1 ]\,\pi_{L-1,N-n_1} [\tilde\eta_2 =n_2 ]\nonumber\\
&\cdots \pi_{L-m+1,N-(n_1 +\ldots+n_{m-1})} [\tilde\eta_m =n_m ]\ .
\label{pf2}
\end{align}
This product form of the marginal distributions enables us to prove the following result by analyzing the scaling behaviour of $Z_{L,N}$ given in the next subsection.\\

\begin{theorem}\label{thm2}
    Consider the ZRP with stationary weights \eqref{eq_weight} in the thermodynamic limit \eqref{thermo} with density $\rho >\rho_c$ \eqref{rhoc} and $\gamma\in (0,2)$. Then we get on the scale $C_L =\sqrt{(\rho -\rho_c )z(1)\theta_L}$ \eqref{scale}
    \[
    \frac{\tilde\eta_1}{C_L} \stackrel{d}{\longrightarrow} \begin{cases}
        0\ &,\mbox{ with probability }\blue{\rho_c /\rho}\\
        Z\ &,\mbox{ with probability }(\rho -\rho_c )/\rho
    \end{cases} \ ,
    \]
    where $Z\sim\Gamma_{2,1}$ is a Gamma-distributed random variable with $\mathbb P[Z\leq u] =\int_0^u se^{-s} ds$.\\
    For any fixed $m\in\N$, $\tilde\eta_1 ,\ldots ,\tilde\eta_m$ converge in distribution to iid random variables with the above marginal distribution.\\
\end{theorem}

\begin{figure}
    \begin{center}
\mbox{\includegraphics[width=0.5\textwidth]{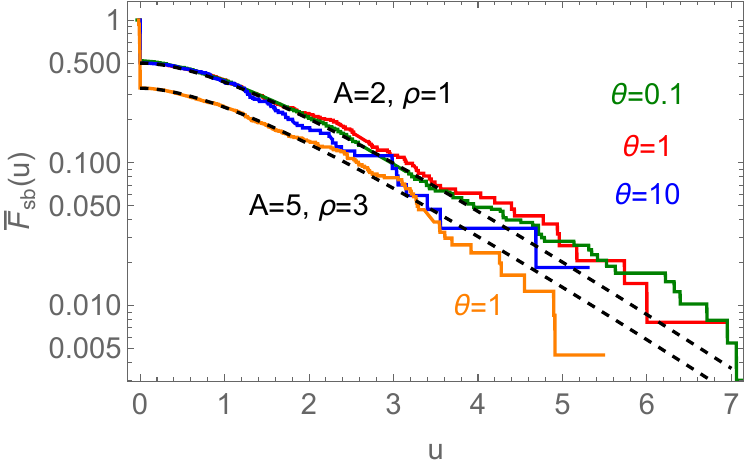}\quad\includegraphics[width=0.5\textwidth]{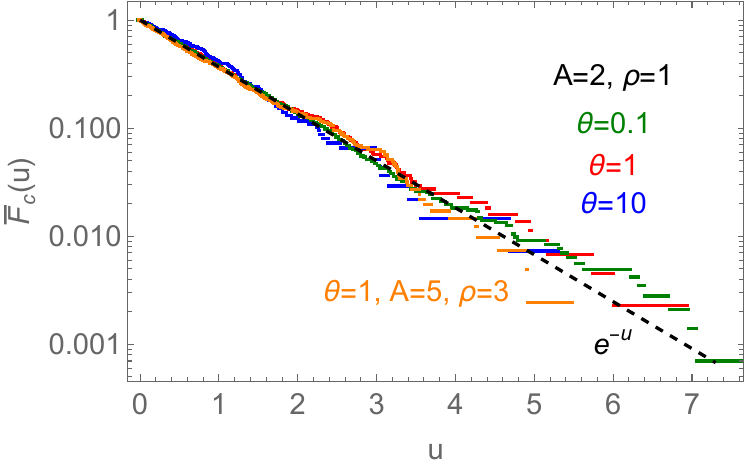}
    }
    \end{center}
    \caption{\label{fig:taila2}
    (Left) The tail of the size-biased empirical distribution $\overline{F}_{sb} (u):=\frac1N\sum_{x\in\Lambda} \eta_x \mathbbm{1}\{\eta_x >uC_L\}$ on scale $C_L$ \eqref{scale} for various values of $\Theta$ (coloured lines) compared to the theoretical prediction of Theorem \ref{thm2} (dashed black line) with $A=5$, $\rho =3$ or $A=2$, $\rho =1$. (Right) The corresponding tail of the conditioned empirical distribution of cluster sites $\overline{F}_c (u):=\sum_{x\in\Lambda} \mathbbm{1}\{\eta_x >uC_L\} \big/\sum_{x\in\Lambda} \mathbbm{1}\{\eta_x \geq A\}$ on scale $C_L$ \eqref{scale} shows exponential decay (dashed black line). \blue{Other parameters are $L=8192$ with $10$ realizations for $A=2$ and $L=4096$ with $100$ realizations for $A=5$.}
    }
\end{figure}

\noindent Note that $\Gamma_{2,1}$ is the size-biased distribution for an exponential random variable with mean $1$, which is therefore the asymptotic distribution of cluster sites consistent with \eqref{heuri}. This is illustrated in Figure \ref{fig:taila2} for the simple example \eqref{simple} with $A=2$. The scales of the cluster sites covered by our result in the limit \eqref{thermo} are of order
\[
C_L \propto \sqrt{\theta_L} \propto L^{\gamma /2} \ll L\quad\mbox{for }\gamma\in (0,2)\ .
\]
Since the condensate contains a mass of order $(\rho - \rho_c )L$, we have a diverging number of cluster sites of order $L^{1-\gamma /2}$ so that any finite collection of them is asymptotically independent.

If $\gamma =2$, clusters contain a non-zero fraction of the total mass and can no longer be independent in size. This borderline case is difficult to analyse exactly and we will get back to it in the discussion Section \ref{sec:conclusion}. For $\gamma >2$, we can show that all the mass in the condensate concentrates on a single cluster site.\\

\begin{theorem}\label{thm3}
    Consider the ZRP with stationary weights \eqref{eq_weight} in the thermodynamic limit \eqref{thermo} with density $\rho >\rho_c$ \eqref{rhoc} and $\gamma >2$. Then
    \[
    \frac{\tilde\eta_1}{(\rho -\rho_c)L} \stackrel{d}{\longrightarrow} \begin{cases}
        0\ &,\mbox{ with probability }\blue{\rho_c /\rho}\\
        1\ &,\mbox{ with probability }(\rho -\rho_c )/\rho
    \end{cases} \ .
    \]
    Furthermore, there is only a single cluster site in the limit,
    \[
    \sum_{x\in\Lambda} \mathbbm{1}\{\eta_x \geq A\}\stackrel{d}{\longrightarrow} 1\quad\mbox{as }L\to\infty\ .
    \]
\end{theorem}

\blue{The first statement is written in analogy to Theorem \ref{thm2} and implies that there is only one macroscopic cluster with mass on scale $L$. The second statement is stronger, stating that all but one site have occupation numbers smaller than $A$.}

\subsection{Proof of the main results}

\subsubsection*{Proof of Theorem \ref{thm1}}
In the limit \eqref{thermo}, we have in particular that the weights $w_L$ \eqref{eq_weight} converge uniformly to $w:\N_0 \to [0,\infty )$ with $w(0),w(1)>0$ and show a sub-exponential decay, i.e.
\[
\frac{1}{L}\log w_L (aL) \simeq -\frac1L \log (\theta_L )\to 0\quad\mbox{as }L\to\infty \mbox{ for all }a>0\ .
\]
Therefore we are in the situation of Appendix A in \cite{chleboun2022poisson}, where Proposition A.1 implies the equivalence of ensembles, i.e. for any finite marginal supported on $M\subset\Lambda$ with $|M|=m$
\begin{equation}\label{equi}
\pi_{L,N}^M \to \begin{cases}
    \nu_\phi^m &,\mbox{ for } R(\phi)=\rho <\rho_c \\
    \nu_1^m &,\mbox{ for }\rho \geq\rho_c
\end{cases}\ .
\end{equation}
Here $\pi_{L,N}$ are the canonical measures \eqref{wpi} and the marginals of the limit measures $\nu_\phi$ derived from \eqref{nu} are
\[
\nu_\phi [\eta_x =k] =\frac{1}{z(\phi )}\phi^k w(k)\ ,\quad k\in\{ 0,\ldots ,A-1\} \quad\mbox{for any fixed }A\geq 1\ ,
\]
with $z(\phi )$ as in \eqref{rzlim}. For $\rho >\rho_c$ and $\phi =1$, this immediately implies Theorem \ref{thm1}.

\enlargethispage{5mm}
\subsubsection*{Proof of Theorem \ref{thm2}}
We can use an analogous approach to Appendix A in \cite{chleboun2022poisson} to derive the asymptotic behaviour of the partition function $Z_{L,N}$.\\

\begin{lemma}\label{lemma3}
In the thermodynamic limit \eqref{thermo} with $\rho >\rho_c$ and $\gamma\in (0,2)$
\begin{equation}
\label{eq:ZLN}
    Z_{L,N} =  (z(1))^L D_L (\rho ) e^{L/\sqrt{\theta_L} f(\rho )} \big(1+o(1)\big)\ ,
\end{equation}
where $f(\rho )=2\sqrt{(\rho -\rho_c )/z(1)}$ and $D_L (\rho )$ is sub-exponential, such that
\begin{equation}\label{mlcond}
\frac{D_L (\rho )}{D_{L-1} (\rho -o(1))} \to 1\quad\mbox{as }L\to\infty\ .    
\end{equation}
\end{lemma}

\begin{proof}
Tilting the stationary weights $w_L$ \eqref{eq_weight} by $\Phi_L (\rho )$ as in \eqref{gcancurr} we can write with \eqref{eq: weights}
\begin{equation}\label{llt}
Z_{L,N} =\nu_{\Phi_L (\rho )}^L \Big[\sum_{x\in\Lambda} \eta_x =N\Big] \Phi_L (\rho )^{-N} z_L \big(\Phi_L (\rho )\big)^L \ .
\end{equation}
\blue{Here $\nu_{\Phi_L (\rho )}^L$ is the grand-canonical product measure \eqref{nu} with fugacity chosen such that the density \eqref{rhol} equals $\rho >\rho_c$. Naturally, $D_L (\rho )=\nu_{\Phi_L (\rho )}^L \Big[\sum_{x\in\Lambda} \eta_x =N\Big]$ decays with increasing system size $L$, in general exponentially fast. But with $N/L\to\rho$ the event $\sum_{x\in\Lambda} \eta_x =N$ is \textit{typical} under the measure $\nu_{\Phi_L (\rho )}^L$, i.e.\ the probability decays significantly slower than exponential in analogy to the standard central limit theorem, where it decays like $1/\sqrt{L}$. In our case, the distribution $\nu_{\Phi_L (\rho )}$ of a single term $\eta_x$ in the sum depends on the number of summands $L$. Still the above heuristic holds and is made precise in a \textit{local limit theorem} for so-called \textit{triangular arrays} (see e.g.\ Lemma A.3 in \cite{chleboun2022poisson} and Theorem 1.2 in \cite{davis1995elementary}), so that $D_L (\rho )$ decays only sub-exponentially with $L$ and in particular \eqref{mlcond} holds.}\\
Using $\Phi_L (\rho )=1-1/C_L$ with $C_L =\sqrt{(\rho -\rho_c )z(1)\theta_L}$ \eqref{scale}, a straightforward asymptotic analysis of \eqref{gcancurr} and \eqref{zlphi} yields
\begin{align*}
\Phi_L (\rho )^{-N} &=\Big( 1-\frac{N}{N\, C_L}\Big)^{-N} \simeq e^{N/C_L} =\exp\Big(\frac{L}{\sqrt{\theta_L}} \frac{\rho}{\sqrt{(\rho -\rho_c )z(1)}}\Big)\\
z_L \big(\Phi_L (\rho )\big)^L &\simeq z(1)^L \Big(1-\frac{\rho_c}{C_L}+\frac{C_L}{\theta_L z(1)}\Big)^L \simeq \exp\Big(\frac{L}{\sqrt{\theta_L}} \frac{\rho -2\rho_c}{\sqrt{(\rho -\rho_c )z(1)}}\Big)
\end{align*}
which implies \eqref{eq:ZLN}.
\end{proof}




\begin{proposition}\label{prop4}
On the scale $C_L =\sqrt{(\rho -\rho_c )\theta_L z(1)}$ \eqref{scale}, we get
\begin{equation}
\pi_{L,N}\left[ \tilde{\eta}_1 \leq C_L u\right] \to\frac{\rho_c}{\rho}+ \frac{\rho-\rho_c}{\rho}  \int_0^u se^{-s}\; ds\ ,
\end{equation}
for all $u>0$ in the thermodynamic limit \eqref{thermo} with $\rho >\rho_c$ and $\gamma\in (0,2)$.
\end{proposition}

\begin{proof}
With \eqref{sb}, we get for any $u>0$
\begin{align*}
\pi_{L,N}\left[ \tilde{\eta}_1 \leq C_L u\right] 
    &= \sum_{n\leq C_L u} \frac{L}{N}  n w_L(n) \frac{Z_{L-1, N-n}}{Z_{L,N}} \\
    &\simeq \frac{1}{\rho} \sum_{n=0}^{A-1} n w(n) \frac{Z_{L-1, N-n}}{Z_{L,N}}+\frac{1}{\rho} \sum_{n=A}^{C_L u} n \frac{1}{\theta_L}  \frac{Z_{L-1, N-n}}{Z_{L,N}} \ .
\end{align*}
With Lemma \ref{lemma3}, we get for $n=O(1)$, $\frac{Z_{L-1, N-n}}{Z_{L,N}}\simeq \frac{1}{z(1)}$, and for $n=C_L s$, $s>0$
\[
\frac{Z_{L-1, N-n}}{Z_{L,N}}\simeq \frac{1}{z(1)}e^{L/\sqrt{\theta_L} (f(\rho -n/L )-f(\rho ))} \simeq \frac{1}{z(1)}e^{-L/\sqrt{\theta_L } f'(\rho) C_L s/L}\simeq \frac{1}{z(1)}e^{-s}
\]
since $f'(\rho)=1/\sqrt{(\rho -\rho_c)z(1)} =\sqrt{\theta_L}/C_L$. Therefore
\[
\pi_{L,N}\left[ \tilde{\eta}_1 \leq C_L u\right] 
    \simeq \frac{\rho_c}{\rho} +\frac{1}{\rho} \frac{C_L^2}{\theta_L z(1)}  \int_0^u se^{-s} ds\to\frac{\rho_c}{\rho}+ \frac{\rho-\rho_c}{\rho}  \int_0^u se^{-s}\, ds\ ,
\]
as $L\to\infty$, finishing the proof.
\end{proof}




\noindent Using the factorization \eqref{pf2} of the joint distribution, we get
\begin{align*}
&\pi_{L,N}\Big[ \tilde{\eta}_1 \leq C_L u_1,\ldots ,\tilde{\eta}_m \leq C_L u_m \Big] =\sum_{n_1 \leq C_L u_1} \!\!\!\ldots\!\!\! \sum_{n_m \leq C_L u_m} \pi_{L,N} [\tilde\eta_1 =n_1 ,\ldots ,\tilde\eta_m =n_m ]\\
&\quad =\sum_{n_1 \leq C_L u_1}\pi_{L,N} [\tilde\eta_1 =n_1 ]\sum_{n_2 \leq C_L u_2} \pi_{L-1,N-n_1} [\tilde\eta_2 =n_2 ]\\
&\qquad\qquad\qquad\qquad\qquad\qquad\qquad \cdots \sum_{n_m \leq C_L u_m}\pi_{L-m+1,N-(n_1 +..+n_{m-1})} [\tilde\eta_m =n_m ]\\
    &\quad\to \left(\frac{\rho_c}{\rho}+ \frac{\rho-\rho_c}{\rho}  \int_0^{u_1} se^{-s}\; ds\right)\cdots \left(\frac{\rho_c}{\rho}+ \frac{\rho-\rho_c}{\rho}  \int_0^{u_m} se^{-s}\; ds\right)\ ,
\end{align*}
since $m$ is finite and $C_L (u_1 +\ldots +u_m) \ll N$ in the limit \eqref{thermo} with $\gamma\in (0,2)$, so we can apply Proposition \ref{prop4} repeatedly in the iterated sum. This concludes the proof of Theorem \ref{thm2}.

\subsubsection*{Proof of Theorem \ref{thm3}}
For $\gamma >2$, we can get the following elementary estimate on the partition function.\\

\begin{lemma}
    In the thermodynamic limit \eqref{thermo} with $\rho >\rho_c$ and $\gamma >2$
    \[
    Z_{L,N} = z(1)^{L-1} \frac{L}{\theta_L} \big( 1+o(1)\big)\ .
    \]
\end{lemma}

\begin{proof}
    Let $K$ denote the number of cluster sites with occupation $\eta_x \geq A$ and $M$ the total mass in the condensed phase, then we can write the partition function in a phase separated form 
    \[
    Z_{L,N} =\sum_{M=0}^N \sum_{K=0}^L {L\choose K} Z^b_{L-K,N-M} Z^c_{K,M} \ ,
    \]
    \blue{where $Z^b_{L-K,N-M}$ and $Z^c_{K,M}$ will be defined and explained below.}
    Here, ${L\choose K}$ is the number of ways to choose $K$ condensed sites within the total volume $|\Lambda |=L$ and there are no ``interface contributions" since the stationary weights \eqref{eq_weight} are of product form. 
For the condensed part $Z^c_{K,M}$, we have constant weights $1/\theta_L$ \blue{for each site with $\eta_x \geq A$ leading to a uniform distribution of mass. Thus,}
\[
Z^c_{K,M} =\frac{1}{\theta_L^K} {M-AK+K-1\choose K-1}\quad\mbox{for }M\geq AK
\]
\blue{where the combinatorial factor} counts the number of possible configurations. From the equivalence of ensembles \eqref{equi} and Theorem \ref{thm1}, we know that in the limit \eqref{thermo} with $\rho >\rho_c$\blue{,} the sum over $M$ is dominated by terms with $M=(\rho -\rho_c )L$, $K\geq 1$ and that the bulk part is asymptotically to leading order
    \[
    Z^b_{L-K,\rho_c L} \simeq z(1)^{L-K}
    \]
    with grand-canonical partition function as given in \eqref{rzlim}. Together we get
    \begin{equation}\label{partk}
    Z_{L,N} \simeq z(1)^{L-1} \frac{L}{\theta_L} +\sum_{K=2}^L \frac{z(1)^{L-K}}{\theta_L^K} R_L (K)
    \end{equation}
    with remainder terms
    \begin{equation}\label{remainder}
    R_L (K)={L\choose K}{N-\rho_c L -AK+K-1\choose K-1}\leq\frac{L^K}{K!}\frac{(\rho -\rho_c )^{K-1} L^{K-1}}{(K-1)!}\ .
    \end{equation}
    Therefore with $L^2 \ll\theta_L$
    \[
    \frac{1}{\theta_L}\sum_{K=2}^L z(1)^{L-K} R_L (K) \leq z(1)^{L-1} \frac{L}{\theta_L} \underbrace{\frac{L^2}{\theta_L} \frac{\rho -\rho_c }{z(1)}}_{\to 0}\underbrace{\sum_{K=0}^\infty \Big(\frac{L^2 (\rho -\rho_c)}{\theta_L z(1)}\Big)^K \frac{1}{K!}}_{=e^{\frac{L^2 (\rho -\rho_c)}{\theta_L z(1)}}\to 1}\ ,
    \]
    which implies\quad $Z_{L,N} \simeq z(1)^{L-1} \frac{L}{\theta_L}$\quad as required.
\end{proof}

Therefore $Z_{L,N}$ is asymptotically dominated by configurations with a single cluster site, so that
\[
\pi_{L,N} \bigg[\sum_{x\in\Lambda}\mathbbm{1}\{\eta_x \geq A\} =1\bigg] =\frac{1}{Z_{L,N}} \sum_{\eta\in E_{L,N}} \prod_{x\in\Lambda} w_L (\eta_x ) \mathbbm{1}\Big\{\sum_{x\in\Lambda}\mathbbm{1}\{\eta_x \geq A\} =1\Big\}\to 1
\]
in the limit \eqref{thermo} with $\rho >\rho_c$ and $\gamma >2$. This proves the second statement of Theorem \ref{thm3} and implies in particular also the first statement.\\

\subsection{Fixed lattice size\label{sec:fixedL}}

Condensation in particle systems has also been studied in a different limit, where the lattice size $L$ remains fixed and only the number of particles $N$ tends to infinity for zero-range processes \cite{ferrari2007escape}, inclusion processes \cite{grosskinsky2011condensation} and other more general models (see e.g. \cite{rafferty2018monotonicity}). This is also the scaling for most dynamical results on metastability  \cite{beltran2012metastability, grosskinsky2013dynamics} and on equilibration dynamics \cite{beltran2017martingale,armendariz2023fluid}. 
In the present model, this scaling regime also leads to condensation when $\theta =\theta_N$ is scaled with the number of particles. Since the density $N/L$ diverges in this regime, the nature of the transition changes and we can have only a single condensate site on scale $N$ as seen below. So the structure of the condensed phase is less interesting than in the thermodynamic limit, but we include a brief outline of this case for completeness.

For a given configuration $\eta$, we let  
\blue{\[ M_N =\max(\eta_x :x\in\Lambda )\in\N\quad\text{and}\quad X_N =\{ x:\eta_x =M_N \}\subset\Lambda \]}be the size and location(s) of the maximum. \blue{In the following, the system parameter $\theta_N$ replaces $\theta_L$ and scales with the diverging number of particles $N$ since $L$ remains fixed.}\\

\begin{theorem}\label{thm4}
    We consider the zero-range process with stationary weights \eqref{eq_weight} in the scaling limit
    \begin{equation}\label{lscaling}
    L\geq 2\mbox{ fixed },\quad N,\theta_N \to\infty\quad\mbox{such that}\quad \frac{\theta_N}{N^\gamma}\to\Theta >0\mbox{ for some }\gamma >0\ .        
    \end{equation}
    Then, for $\gamma >1$, we have condensation of the full mass fraction on a single uniformly chosen site
    \[
        \frac1N M_N \stackrel{d}{\longrightarrow} 1 \quad\mbox{and}\quad X_N \stackrel{d}{\longrightarrow}\{ X\},\ X\in\Lambda\mbox{ uniform }\ ,
    \]
    and the bulk sites converge to iid variables with distribution (cf.\ Theorem \ref{thm1})
    \[
    \pi_{L,N} [\eta_x =n_x :x\in\Lambda\setminus X_N ]\to \frac{1}{z(1)^{L-1}}\prod_{x\in\Lambda\setminus X_N} w(n_x )\quad\mbox{for all }n_x \in\{ 0,\ldots ,A-1\}\ ,
    \]
    where as before $z(1)=w(0)+\ldots +w(A-1).$\\

For $\gamma <1$, all occupation numbers diverge on the scale $N$ with
\begin{equation}\label{zxlim}
\frac{\eta_x}{N}\stackrel{d}{\longrightarrow} Z_x \in [0,1]\quad\mbox{with}\quad\mathbb P [Z_x \leq u]=1-(1-u)^{L-1}\ .
\end{equation}
We have $\sum_{x\in\Lambda} Z_x =1$, so the limits are not independent.\\
\end{theorem}

\noindent \blue{Note that for $\gamma =1$ we cannot obtain a rigorous result and this case is discussed below in the penultimate paragraph of this subsection.} To prove this result, we first find the asymptotic behaviour of $Z_{L,N}$ \eqref{eq:ZLN} by direct computation and explicit estimates.\\

\begin{lemma}\label{lemma8}
    In the scaling limit \eqref{lscaling} for any fixed $L\geq 2$ the partition function \eqref{eq:ZLN} is asymptotically given by
    \[
    Z_{L,N} \simeq \left\{\begin{array}{cl}
       \frac{1}{\theta_N} L\, z(1)^{L-1}  &,\ \theta_N \gg N \ (\text{i.e. }\gamma >1) \\
       \frac{N^{L-1}}{\theta_N^L} \frac{1}{(L-1)!} &,\ \theta_N \ll N \ (\text{i.e. }\gamma <1)
    \end{array}\right.\quad\mbox{as }N\to\infty\ .
    \]
\end{lemma}

\begin{proof}
    For $N$ large enough, there has to be at least one cluster site with $\eta_x \geq A$ and we have in analogy to \eqref{partk} and \eqref{remainder}
    \begin{equation}\label{eq:pfsum}
Z_{L,N} =z(1)^L \sum_{K=1}^L \frac{1}{(z(1)\theta_N )^K} {L\choose K} \frac{N^{K-1}}{(K-1)!}
    \end{equation}
    where $L$ is now fixed. It is clear that for $\theta_N \gg N$ the first term dominates the sum since
    \[
    Z_{L,N} \simeq\frac{z(1)^{L-1}}{\theta_N} \bigg( L+\sum_{K=1}^{L-1} \Big(\frac{N}{\theta_N}\Big)^K {L\choose K+1} \frac{1}{z(1)^K K!}\bigg)\ .
    \]
    For $\theta_N \ll N$ the last term dominates with
    \[
    Z_{L,N} \simeq\frac{N^{L-1}}{\theta_N^L} \bigg( \frac{1}{(L-1)!}+\sum_{K=1}^{L-1} \Big(\frac{\theta_N}{N}\Big)^K {L\choose L-K} \frac{1}{z(1)^K (L-K-1)!}\bigg)\ ,
    \]
    and this implies the statement.
\end{proof}

\textbf{$\boldsymbol{\gamma >1}$.} Since the partition function is dominated by a single cluster site we have
\begin{align*}
\pi_{L,N} \big[ |X_N|\geq 2\big] &=\sum_{n=2A}^N \frac{1}{\theta_N^2} \frac{Z_{L-2,N-n}}{Z_{L,N}} =\frac{1}{\theta_N^2} \bigg(\sum_{n=2A}^{N-(L-2)A} +\sum_{n=N-(L-2)A+1}^N \bigg)\frac{Z_{L-2,N-n}}{Z_{L,N}}\\
&\leq \frac{1}{\theta_N^2}\frac{L-2}{Lz(1)^2} N +\frac{1}{\theta_N^2}(L-2)A C_L \theta_N \leq \frac{C_L}{\theta_N} \to 0\mbox{ as }N\to\infty\ .
\end{align*}
Here we split the sum such that in the first part the ratio of partition functions is of order $1$ since both contain at least one cluster site, and in the second part the ratio is bounded by a multiple of $(1/\theta_N )^{-1}$ with the above lemma. 
By symmetry we therefore have $\pi_{L,N} [\eta_y =M_N ]=\frac1L (1+o(1))$ for any $y\in \Lambda$, and with Lemma \ref{lemma8} we get for all $n_x \in\{ 0,\ldots ,A-1\}$
\begin{align*}
&\pi_{L,N} \Big[\eta_x =n_x :x\in\Lambda\setminus\{ y\}\Big|\eta_y =M_N \Big] =\frac{1}{\theta_N}\frac{1}{Z_{L,N}}\Big(\prod_{x\neq y} w(n_x )\Big) \frac{1}{\pi_{L,N} [\eta_y =M_N ]}\\
&\qquad \simeq \frac{1}{L\pi_{L,N} [\eta_y =M_N ]}\frac{1}{z(1)^{L-1}} \prod_{x\neq y} w(n_x )\simeq \frac{1}{z(1)^{L-1}} \prod_{x\neq y} w(n_x )
\end{align*}
which implies the result.\\

\textbf{$\boldsymbol{\gamma <1}$.} For any $u<1$ and $n\leq uN$ we have $\frac{Z_{L-1,N-n}}{Z_{L,N}}\simeq \theta_N \frac{(N-n)^{L-2} (L-1)!}{(L-2)!N^{L-1}}$ with Lemma \ref{lemma8}, and therefore
\begin{align*}
    \pi_{L,N} \big[\eta_x \leq uN \big] &\simeq\sum_{n=0}^{A-1} w(n)\frac{Z_{L-1,N-n}}{Z_{L,N}} +\sum_{n=A}^{\lfloor uN\rfloor} \frac{Z_{L-1,N-n}}{\theta_N Z_{L,N}} \\
    &\simeq z(1)(L-1)\frac{\theta_N}{N} +(L-1)\int_0^u (1-s)^{L-2} ds \to 1-(1-u)^{L-1}
\end{align*}
as $N\to\infty$, which completes the proof of Theorem \ref{thm4}.\\

\begin{figure}
    \begin{center}
\mbox{\includegraphics[width=0.55\textwidth]{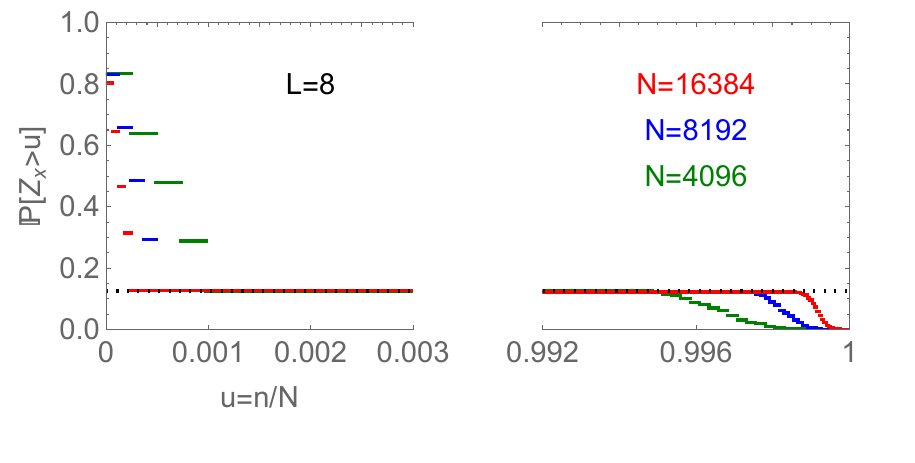}
\includegraphics[width=0.45\textwidth]{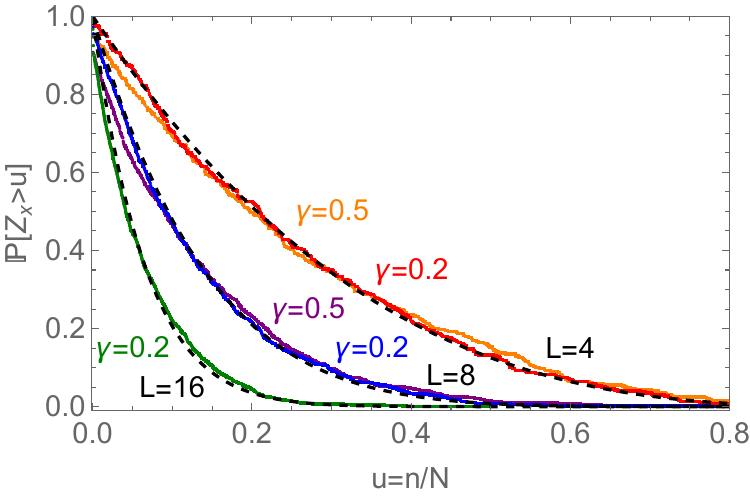}
    }
    \end{center}
    \caption{\label{fig:fixedl}
    For a ZRP with rates \eqref{eq_rates}, $A=2$ and $\rho_c =1/2$ in the scaling regime \eqref{lscaling} with $\theta_N =N^\gamma$ we show the marginal tail distribution of a re-scaled occupation number $Z_x =\eta_x /N$. For $\gamma =2 >1$ (left) the mass condenses on a single lattice site and the distribution shows a plateau at level $1/L$ and a uniform distribution for occupation numbers up to $A$. For $\gamma <1$ (right) the empirical tail distributions of $Z_x$ (coloured lines) fit well with the limiting distributions \eqref{zxlim} (dashed black lines) for various values of $L$ and $\gamma$, where $N=8192$. \blue{In each case $100$ realizations were used.}
    }
\end{figure}

The limiting distribution of $Z_x :=\lim_{N\to\infty} \eta_x /N$ corresponds essentially to a marginal of the uniform distribution of the total mass on the simplex with $Z_1 +\ldots +Z_L =1$, with a leading order correction $z(1)(L-1)\frac{\theta_N}{N}$ corresponding to mass on bulk sites with occupation numbers smaller than $A$. As is shown in Figure \ref{fig:fixedl} (right), for moderate size $N$ this is well visible in simulation results. For the condensing case $\gamma >1$ the empirical distribution of the same variables shows a plateau at level $1/L$, corresponding to the volume fraction of the condensate (see Figure \ref{fig:fixedl} (left)).

Simulation results for the intermediate case $\gamma =1$ show a combination of both cases, with \blue{a} non-zero probability of observing several cluster sites in the condensed phase. Since all terms in \eqref{eq:pfsum} contribute to the partition function \blue{to leading order,} a conclusive analysis is difficult and a simple asymptotic expression cannot be obtained. \blue{The condensed mass is shared uniformly among the cluster sites, and the number of sites is random and changes in time and between realizations. Therefore, stationary empirical tails, as shown in Figure \ref{fig:fixedl} for $\gamma\neq 1$, show a non-specific generic decay in this case, which we chose not to include. 
Since it does not lead to any generic behaviour, we do not discuss the case $\gamma =1$ any further and rather consider the dependence of the condensate statistics on modifications of the jump rates in the Discussion below.}

\blue{For $\gamma >1$, the unique condensate location will exhibit a metastable motion on the lattice $\Lambda$ in analogy to previous results \cite{beltran2012metastability,seo2019condensation} for condensing ZRPs. The time scale for this motion will be dominated by the time it takes to create a second condensate site from bulk sites, which happens when a site with $\eta_x =A$ gains another particle and reaches the stable occupation number $A+1$. A simple argument shows that to leading order this happens after a time of order $\theta_N$. Then both condensate sites compete for particles in analogy to a symmetric random walk with absorbing boundary conditions, and only with probability $1/N$ will the new condensate site win, resulting in an effective motion. Therefore, we expect the metastable motion to occur on the scale $N\theta_N \gg N^2$. The precise motion will depend on the geometry of the lattice, and it would be an interesting question to establish this rigorously applying results in \cite{beltran2015martingale,landim2019metastable} or a recent approach on the $\Gamma$-expansion of large deviation rate functions in \cite{choi2024}.}\\


\section{Discussion\label{sec:conclusion}}

We have shown that the condensed phase in zero-range processes of the form \eqref{eq_rates} with a fast rate $\theta_L \to\infty$ exhibit a condensation transition (Theorem \ref{thm1}). As long as the fast rate scales like $L^\gamma$ with $\gamma\in (0,2)$, the condensed phase consists of a diverging number of independent, exponentially distributed clusters on the scale $C_L \propto \sqrt{\theta_L}$ \eqref{scale} as shown in Theorem \ref{thm2}. For $\gamma >2$ this scale $C_L \gg L$ would exceed the scale of the total mass in the system $N=\rho L$, and the condensate consists only of a single cluster site (Theorem \ref{thm3}). 

Our main result Theorem \ref{thm2} relies on the representation \eqref{llt} of the canonical partition function, where $\nu_{\Phi_L (\rho )}^L \Big[\sum_{x\in\Lambda} \eta_x =N\Big]$ is the probability of a typical event. For $\gamma <2$ each term in the sum $\sum_{x\in\Lambda} \eta_x$ is asymptotically negligible so that the Lindeberg condition is fulfilled and the local limit theorem implies that the above probability is proportional to the inverse standard deviation of the sum. For $\gamma\geq 2$ cluster sites are of size $O(N)$ and make a non-vanishing contribution to the sum, so this argument no longer applies. For $\gamma >2$ we can use elementary estimates to see that asymptotically there is only a single cluster site (Theorem \ref{thm3}). For the boundary case $\gamma =2$ the condensed phase consists of several clusters of macroscopic size $O(N)$ which can be confirmed by simulations, but a rigorous analysis is difficult since neither of the above methods applies.

\begin{figure}
    \begin{center}
\mbox{\includegraphics[width=0.5\textwidth]{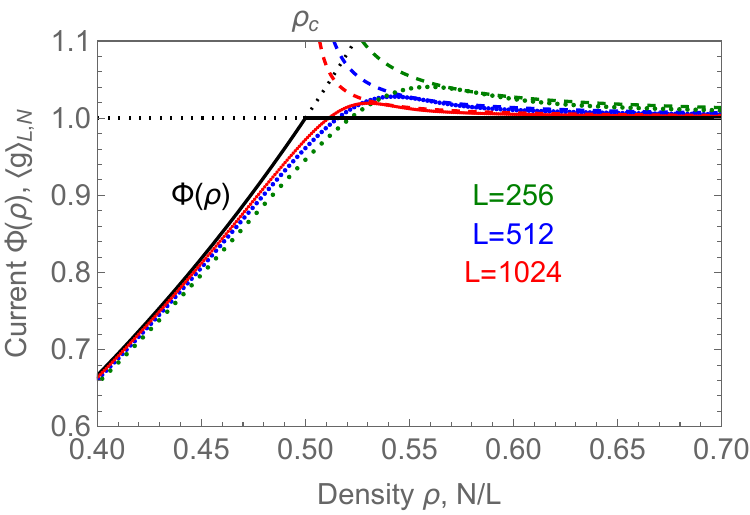}
\ \includegraphics[width=0.5\textwidth]{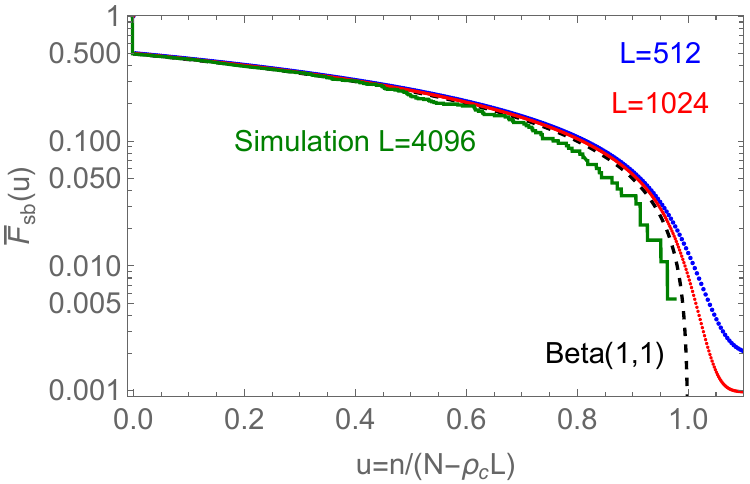}
    }
    \end{center}
    \caption{\label{fig:tailpd}
    For a ZRP with rates \eqref{pdzrp}, $A=2$ and $\rho_c =1/2$, the recursively computed canonical current \eqref{cancurr} (coloured dotted lines) and the grand-canonical current \eqref{gcancurr} (coloured dashed lines) exhibit an overshoot compared to the asymptotic current (full black line) as seen on the left. On the right for $\rho=1$ the size-biased tail distributions $\overline{F}_{sb} (u)$ from recursively computed tails (dotted blue and red lines, analogous to \eqref{cancurr}) and empirical tails from simulations (green line) compare well with an asymptotic Beta$(1,1)$ distribution (black dashed line), as is expected for Poisson-Dirichlet statistics with parameter $1$ (see \cite{chleboun2022poisson} and \cite{feng2010poisson} for details).
    }
\end{figure}

Our results do not depend on the bulk dynamics of the system and hold for general stationary weights $w(n)$ \blue{for} $n=0,\ldots ,A-1$. \blue{The fast rate creates an effective (soft) exclusion interaction for the unstable cluster size $A$, and it takes on the order of $\theta_L$ time for a bulk site to turn into a condensed site with $\eta_x \geq A$. Condensed sites are then free to grow arbitrarily large since they lose particles only at rate $1$, and at stationarity a balance of currents between bulk and cluster sites is established. The number of condensed clusters (and thus their stationary size) clearly depends on $\theta_L$, the slower $\theta_L$ grows with $L$ the easier it is to create a new condensed site. The balance of currents depends on the density $\rho$, as is discussed in detail in the heuristic Section \ref{sec:heuristics}. Note that condensation in this model is driven by particles being expelled from the bulk filled beyond its capacity $\rho_c$, rather than particles attracting each other as previously studied in ZRPs with decreasing jump rates \cite{o1998jamming,evans2000phase,jeon2000size,grosskinsky2003condensation,godreche2003dynamics}.}

It is also not crucial that fast rates occur only for a single occupation number, and our analysis generalizes directly to more occupation numbers with a resulting effective scale given by the inverse product of those rates in analogy to \eqref{eq_weight}. \blue{More precisely, if we consider diverging rates $\theta_L^{(1)} ,\ldots ,\theta_L^{(k)}$ for occupation numbers $A_1 <\ldots <A_k$ then the stationary weights $w_L (n)$ of condensed clusters with $n\geq A_k$ scale like $\big(\prod_{i=1}^k \theta_L^{(i)} )^{-1}$. The same heuristic as above for condensation applies and our rigorous results can be generalized directly, albeit with significantly more complicated notation. In a well-established mapping \cite{kafri2002criterion}, where occupation numbers in the ZRP describe distances between particles in an exclusion model, the diverging rates can be interpreted as a finite-range intervention mechanism: If the distance between particles in the exclusion model gets large ($\geq A_1$) the distance decreases again with a high rate. If in spite of intervention by rare fluctuation the distance gets too large (${>}A_k$) then the mechanism now longer applies ("the next particle is out of sight"). Condensation in the ZRP then corresponds to large separations between particles in the exclusion model.}

On the other hand, the structure of the condensed phase depends rather strongly on \blue{the jump rates for occupation numbers above $A$ (or $A_k$), and here a systematic understanding is still lacking and a very interesting open question resulting from this work.} For example, in \cite{chleboun2022poisson} Section 5, a zero-range process with rates of the form
\begin{equation}\label{pdzrp}
g_L(n)  =
\begin{cases}
1 &\quad\text{ if } n=1,\ldots, A-1\, ,\\
\theta_L &\quad \text{ if } n=A\, ,\\
n/(n-1) &\quad \text{ if }n>A\, ,
\end{cases}    
\end{equation}
has been considered, which leads to stationary weights with $w_L (n)=\frac{1}{\theta_L n}$ for $n\geq A$ such that their size-biased version $nw_L (n)$ is uniform. This system is then in a general class of models that exhibits macroscopic clusters with Poisson-Dirichlet statistics in the scaling $\theta_L /L\to\Theta\geq 0$ (see \cite{chleboun2022poisson} for details), whereas in the present model we would see iid clusters on the scale $O(\sqrt{L})$. \blue{Thus,} the rather innocent change in the jump rates (we still have $g_L (n)\to 1$ for $n\to\infty$) leads to a strong change in the statistics of the condensed phase. In terms of the birth-death heuristics for cluster sites (cf. Section \ref{sec:heuristics}) this is even counter-intuitive, since in a ZRP with rates \eqref{pdzrp} cluster sites lose particles with rates $g_L (n)=\frac{n}{n-1}>1$ but still grow to a larger macroscopic scale. This is the result of an overshoot effect in the stationary current of the system for supercritical densities as is illustrated in Figure \ref{fig:tailpd}, so that the rate of incoming particles into cluster sites is also increased. This overshoot is rather common and has been observed in many other condensing zero-range models \cite{evans2006canonical,chleboun2010finite,rafferty2018monotonicity}. In the above system, this leads indeed to Poisson-Dirichlet statistics for the cluster sites which is consistent with an asymptotic Beta distribution for the empirical measure of cluster sites (see \cite{feng2010poisson} for details).

In general, relations \eqref{pf} and \eqref{pf2} hold for any interacting particle system with stationary product measures. They constitute a microscopic version of a \blue{stick-breaking} process for residual allocation models (see e.g.\ \cite{feng2021note} and references therein), which have so far mostly been used in statistics and population genetics \cite{feng2010poisson}. Together with the general representation \eqref{llt}, this is a promising starting point to establish a systematic understanding of the structure of the condensed phase in condensing particle systems with stationary product measures.\\

\section*{Acknowledgement}

The authors are grateful for useful discussions with Simon Gabriel and Paul Chleboun, and for preliminary work on simulations by Joshua Blank. This work (Grant No. RGNS 63-180) was supported by Office of the Permanent Secretary, Ministry of Higher Education, Science, Research and Innovation  (OPS MHESI), Thailand Science Research and Innovation (TSRI) and Mahidol University.\\

Data sets generated for simulations are available from the corresponding author on reasonable request.\\


\begin{thebibliography}{42}
	\ifx \bisbn   \undefined \def \bisbn  #1{ISBN #1}\fi
	\ifx \binits  \undefined \def \binits#1{#1}\fi
	\ifx \bauthor  \undefined \def \bauthor#1{#1}\fi
	\ifx \batitle  \undefined \def \batitle#1{#1}\fi
	\ifx \bjtitle  \undefined \def \bjtitle#1{#1}\fi
	\ifx \bvolume  \undefined \def \bvolume#1{\textbf{#1}}\fi
	\ifx \byear  \undefined \def \byear#1{#1}\fi
	\ifx \bissue  \undefined \def \bissue#1{#1}\fi
	\ifx \bfpage  \undefined \def \bfpage#1{#1}\fi
	\ifx \blpage  \undefined \def \blpage #1{#1}\fi
	\ifx \burl  \undefined \def \burl#1{\textsf{#1}}\fi
	\ifx \doiurl  \undefined \def \doiurl#1{\url{https://doi.org/#1}}\fi
	\ifx \betal  \undefined \def \betal{\textit{et al.}}\fi
	\ifx \binstitute  \undefined \def \binstitute#1{#1}\fi
	\ifx \binstitutionaled  \undefined \def \binstitutionaled#1{#1}\fi
	\ifx \bctitle  \undefined \def \bctitle#1{#1}\fi
	\ifx \beditor  \undefined \def \beditor#1{#1}\fi
	\ifx \bpublisher  \undefined \def \bpublisher#1{#1}\fi
	\ifx \bbtitle  \undefined \def \bbtitle#1{#1}\fi
	\ifx \bedition  \undefined \def \bedition#1{#1}\fi
	\ifx \bseriesno  \undefined \def \bseriesno#1{#1}\fi
	\ifx \blocation  \undefined \def \blocation#1{#1}\fi
	\ifx \bsertitle  \undefined \def \bsertitle#1{#1}\fi
	\ifx \bsnm \undefined \def \bsnm#1{#1}\fi
	\ifx \bsuffix \undefined \def \bsuffix#1{#1}\fi
	\ifx \bparticle \undefined \def \bparticle#1{#1}\fi
	\ifx \barticle \undefined \def \barticle#1{#1}\fi
	\bibcommenthead
	\ifx \bconfdate \undefined \def \bconfdate #1{#1}\fi
	\ifx \botherref \undefined \def \botherref #1{#1}\fi
	\ifx \url \undefined \def \url#1{\textsf{#1}}\fi
	\ifx \bchapter \undefined \def \bchapter#1{#1}\fi
	\ifx \bbook \undefined \def \bbook#1{#1}\fi
	\ifx \bcomment \undefined \def \bcomment#1{#1}\fi
	\ifx \oauthor \undefined \def \oauthor#1{#1}\fi
	\ifx \citeauthoryear \undefined \def \citeauthoryear#1{#1}\fi
	\ifx \endbibitem  \undefined \def \endbibitem {}\fi
	\ifx \bconflocation  \undefined \def \bconflocation#1{#1}\fi
	\ifx \arxivurl  \undefined \def \arxivurl#1{\textsf{#1}}\fi
	\csname PreBibitemsHook\endcsname
	
	\bibitem[\protect\citeauthoryear{O’loan et~al.}{1998}]{o1998jamming}
	\begin{barticle}
		\bauthor{\bsnm{O’loan}, \binits{O.}},
		\bauthor{\bsnm{Evans}, \binits{M.R.}},
		\bauthor{\bsnm{Cates}, \binits{M.E.}}:
		\batitle{Jamming transition in a homogeneous one-dimensional system: The bus
			route model}.
		\bjtitle{Physical Review E}
		\bvolume{58}(\bissue{2}),
		\bfpage{1404}
		(\byear{1998})
	\end{barticle}
	\endbibitem
	
	\bibitem[\protect\citeauthoryear{Evans}{2000}]{evans2000phase}
	\begin{barticle}
		\bauthor{\bsnm{Evans}, \binits{M.R.}}:
		\batitle{Phase transitions in one-dimensional nonequilibrium systems}.
		\bjtitle{Brazilian Journal of Physics}
		\bvolume{30}(\bissue{1}),
		\bfpage{42}--\blpage{57}
		(\byear{2000})
	\end{barticle}
	\endbibitem
	
	\bibitem[\protect\citeauthoryear{Jeon et~al.}{2000}]{jeon2000size}
	\begin{barticle}
		\bauthor{\bsnm{Jeon}, \binits{I.}},
		\bauthor{\bsnm{March}, \binits{P.}},
		\bauthor{\bsnm{Pittel}, \binits{B.}}:
		\batitle{Size of the largest cluster under zero-range invariant measures}.
		\bjtitle{The Annals of Probability}
		\bvolume{28}(\bissue{3}),
		\bfpage{1162}--\blpage{1194}
		(\byear{2000})
	\end{barticle}
	\endbibitem
	
	\bibitem[\protect\citeauthoryear{Grosskinsky
		et~al.}{2003}]{grosskinsky2003condensation}
	\begin{barticle}
		\bauthor{\bsnm{Grosskinsky}, \binits{S.}},
		\bauthor{\bsnm{Sch{\"u}tz}, \binits{G.M.}},
		\bauthor{\bsnm{Spohn}, \binits{H.}}:
		\batitle{Condensation in the zero range process: stationary and dynamical
			properties}.
		\bjtitle{Journal of statistical physics}
		\bvolume{113},
		\bfpage{389}--\blpage{410}
		(\byear{2003})
	\end{barticle}
	\endbibitem
	
	\bibitem[\protect\citeauthoryear{Godr{\`e}che}{2003}]{godreche2003dynamics}
	\begin{barticle}
		\bauthor{\bsnm{Godr{\`e}che}, \binits{C.}}:
		\batitle{Dynamics of condensation in zero-range processes}.
		\bjtitle{Journal of Physics A: Mathematical and General}
		\bvolume{36}(\bissue{23}),
		\bfpage{6313}
		(\byear{2003})
	\end{barticle}
	\endbibitem
	
	\bibitem[\protect\citeauthoryear{Grosskinsky
		et~al.}{2011}]{grosskinsky2011condensation}
	\begin{barticle}
		\bauthor{\bsnm{Grosskinsky}, \binits{S.}},
		\bauthor{\bsnm{Redig}, \binits{F.}},
		\bauthor{\bsnm{Vafayi}, \binits{K.}}:
		\batitle{Condensation in the inclusion process and related models}.
		\bjtitle{Journal of Statistical Physics}
		\bvolume{142},
		\bfpage{952}--\blpage{974}
		(\byear{2011})
	\end{barticle}
	\endbibitem
	
	\bibitem[\protect\citeauthoryear{Evans and
		Waclaw}{2014}]{evans2014condensation}
	\begin{barticle}
		\bauthor{\bsnm{Evans}, \binits{M.R.}},
		\bauthor{\bsnm{Waclaw}, \binits{B.}}:
		\batitle{Condensation in stochastic mass transport models: beyond the
			zero-range process}.
		\bjtitle{Journal of Physics A: Mathematical and Theoretical}
		\bvolume{47}(\bissue{9}),
		\bfpage{095001}
		(\byear{2014})
	\end{barticle}
	\endbibitem
	
	\bibitem[\protect\citeauthoryear{Beltr{\'a}n
		et~al.}{2017}]{beltran2017martingale}
	\begin{barticle}
		\bauthor{\bsnm{Beltr{\'a}n}, \binits{J.}},
		\bauthor{\bsnm{Jara}, \binits{M.}},
		\bauthor{\bsnm{Landim}, \binits{C.}}:
		\batitle{A martingale problem for an absorbed diffusion: the nucleation phase
			of condensing zero range processes}.
		\bjtitle{Probability Theory and Related Fields}
		\bvolume{169},
		\bfpage{1169}--\blpage{1220}
		(\byear{2017})
	\end{barticle}
	\endbibitem
	
	\bibitem[\protect\citeauthoryear{Godr{\`e}che and
		Drouffe}{2016}]{godreche2016coarsening}
	\begin{barticle}
		\bauthor{\bsnm{Godr{\`e}che}, \binits{C.}},
		\bauthor{\bsnm{Drouffe}, \binits{J.-M.}}:
		\batitle{Coarsening dynamics of zero-range processes}.
		\bjtitle{Journal of Physics A: Mathematical and Theoretical}
		\bvolume{50}(\bissue{1}),
		\bfpage{015005}
		(\byear{2016})
	\end{barticle}
	\endbibitem
	
	\bibitem[\protect\citeauthoryear{Jatuviriyapornchai and
		Grosskinsky}{2016}]{jatuviriyapornchai2016coarsening}
	\begin{barticle}
		\bauthor{\bsnm{Jatuviriyapornchai}, \binits{W.}},
		\bauthor{\bsnm{Grosskinsky}, \binits{S.}}:
		\batitle{Coarsening dynamics in condensing zero-range processes and size-biased
			birth death chains}.
		\bjtitle{Journal of Physics A: Mathematical and Theoretical}
		\bvolume{49}(\bissue{18}),
		\bfpage{185005}
		(\byear{2016})
	\end{barticle}
	\endbibitem
	
	\bibitem[\protect\citeauthoryear{Armend{\'a}riz
		et~al.}{2023}]{armendariz2023fluid}
	\begin{botherref}
		\oauthor{\bsnm{Armend{\'a}riz}, \binits{I.}},
		\oauthor{\bsnm{Beltr{\'a}n}, \binits{J.}},
		\oauthor{\bsnm{Cuesta}, \binits{D.}},
		\oauthor{\bsnm{Jara}, \binits{M.}}:
		Fluid limit for the coarsening phase of the condensing zero-range process.
		arXiv preprint arXiv:2302.05497
		(2023)
	\end{botherref}
	\endbibitem
	
	\bibitem[\protect\citeauthoryear{Beltr{\'a}n and
		Landim}{2012}]{beltran2012metastability}
	\begin{barticle}
		\bauthor{\bsnm{Beltr{\'a}n}, \binits{J.}},
		\bauthor{\bsnm{Landim}, \binits{C.}}:
		\batitle{Metastability of reversible condensed zero range processes on a finite
			set}.
		\bjtitle{Probability Theory and Related Fields}
		\bvolume{152}(\bissue{3}),
		\bfpage{781}--\blpage{807}
		(\byear{2012})
	\end{barticle}
	\endbibitem
	
	\bibitem[\protect\citeauthoryear{Armend{\'a}riz
		et~al.}{2015}]{armendariz2015metastability}
	\begin{botherref}
		\oauthor{\bsnm{Armend{\'a}riz}, \binits{I.}},
		\oauthor{\bsnm{Grosskinsky}, \binits{S.}},
		\oauthor{\bsnm{Loulakis}, \binits{M.}}:
		Metastability in a condensing zero-range process in the thermodynamic limit.
		Probability Theory and Related Fields,
		1--71
		(2015)
	\end{botherref}
	\endbibitem
	
	\bibitem[\protect\citeauthoryear{Seo}{2019}]{seo2019condensation}
	\begin{barticle}
		\bauthor{\bsnm{Seo}, \binits{I.}}:
		\batitle{Condensation of non-reversible zero-range processes}.
		\bjtitle{Communications in Mathematical Physics}
		\bvolume{366},
		\bfpage{781}--\blpage{839}
		(\byear{2019})
	\end{barticle}
	\endbibitem
	
	\bibitem[\protect\citeauthoryear{Choi}{2024}]{choi2024}
	\begin{botherref}
		\oauthor{\bsnm{Choi}, \binits{K.}}:
		A $\gamma$-convergence of level-two large deviation for metastable systems: The
		case of zero-range processes.
		arXiv preprint arXiv:2405.10631
		(2024)
	\end{botherref}
	\endbibitem
	
	\bibitem[\protect\citeauthoryear{Grosskinsky
		et~al.}{2013}]{grosskinsky2013dynamics}
	\begin{barticle}
		\bauthor{\bsnm{Grosskinsky}, \binits{S.}},
		\bauthor{\bsnm{Redig}, \binits{F.}},
		\bauthor{\bsnm{Vafayi}, \binits{K.}}:
		\batitle{{Dynamics of condensation in the symmetric inclusion process}}.
		\bjtitle{Electronic Journal of Probability}
		\bvolume{18},
		\bfpage{1}--\blpage{23}
		(\byear{2013})
	\end{barticle}
	\endbibitem
	
	\bibitem[\protect\citeauthoryear{Bianchi
		et~al.}{2017}]{bianchi2017metastability}
	\begin{botherref}
		\oauthor{\bsnm{Bianchi}, \binits{A.}},
		\oauthor{\bsnm{Dommers}, \binits{S.}},
		\oauthor{\bsnm{Giardin{\`a}}, \binits{C.}}:
		Metastability in the reversible inclusion process.
		Electronic Journal of Probability
		\textbf{22}(70)
		(2017)
	\end{botherref}
	\endbibitem
	
	\bibitem[\protect\citeauthoryear{Kim and Seo}{2021}]{kim2021condensation}
	\begin{barticle}
		\bauthor{\bsnm{Kim}, \binits{S.}},
		\bauthor{\bsnm{Seo}, \binits{I.}}:
		\batitle{Condensation and metastable behavior of non-reversible inclusion
			processes}.
		\bjtitle{Communications in Mathematical Physics}
		\bvolume{382},
		\bfpage{1343}--\blpage{1401}
		(\byear{2021})
	\end{barticle}
	\endbibitem
	
	\bibitem[\protect\citeauthoryear{Beltr{\'a}n and
		Landim}{2015}]{beltran2015martingale}
	\begin{barticle}
		\bauthor{\bsnm{Beltr{\'a}n}, \binits{J.}},
		\bauthor{\bsnm{Landim}, \binits{C.}}:
		\batitle{A martingale approach to metastability}.
		\bjtitle{Probability Theory and Related Fields}
		\bvolume{161},
		\bfpage{267}--\blpage{307}
		(\byear{2015})
	\end{barticle}
	\endbibitem
	
	\bibitem[\protect\citeauthoryear{Landim}{2019}]{landim2019metastable}
	\begin{barticle}
		\bauthor{\bsnm{Landim}, \binits{C.}}:
		\batitle{Metastable markov chains}.
		\bjtitle{Probab. Surveys}
		\bvolume{16},
		\bfpage{143}--\blpage{227}
		(\byear{2019})
	\end{barticle}
	\endbibitem
	
	\bibitem[\protect\citeauthoryear{Grosskinsky and
		Sch{\"u}tz}{2008}]{grosskinsky2008discontinuous}
	\begin{barticle}
		\bauthor{\bsnm{Grosskinsky}, \binits{S.}},
		\bauthor{\bsnm{Sch{\"u}tz}, \binits{G.M.}}:
		\batitle{Discontinuous condensation transition and nonequivalence of ensembles
			in a zero-range process}.
		\bjtitle{Journal of statistical physics}
		\bvolume{132}(\bissue{1}),
		\bfpage{77}--\blpage{108}
		(\byear{2008})
	\end{barticle}
	\endbibitem
	
	\bibitem[\protect\citeauthoryear{Chleboun and
		Grosskinsky}{2015}]{chleboun2015dynamical}
	\begin{barticle}
		\bauthor{\bsnm{Chleboun}, \binits{P.}},
		\bauthor{\bsnm{Grosskinsky}, \binits{S.}}:
		\batitle{A dynamical transition and metastability in a size-dependent
			zero-range process}.
		\bjtitle{Journal of Physics A: Mathematical and Theoretical}
		\bvolume{48}(\bissue{5}),
		\bfpage{055001}
		(\byear{2015})
	\end{barticle}
	\endbibitem
	
	\bibitem[\protect\citeauthoryear{Godr{\`e}che and
		Luck}{2012}]{godreche2012condensation}
	\begin{barticle}
		\bauthor{\bsnm{Godr{\`e}che}, \binits{C.}},
		\bauthor{\bsnm{Luck}, \binits{J.-M.}}:
		\batitle{Condensation in the inhomogeneous zero-range process: an interplay
			between interaction and diffusion disorder}.
		\bjtitle{Journal of Statistical Mechanics: Theory and Experiment}
		\bvolume{2012}(\bissue{12}),
		\bfpage{12013}
		(\byear{2012})
	\end{barticle}
	\endbibitem
	
	\bibitem[\protect\citeauthoryear{Schwarzkopf
		et~al.}{2008}]{schwarzkopf2008zero}
	\begin{barticle}
		\bauthor{\bsnm{Schwarzkopf}, \binits{Y.}},
		\bauthor{\bsnm{Evans}, \binits{M.R.}},
		\bauthor{\bsnm{Mukamel}, \binits{D.}}:
		\batitle{Zero-range processes with multiple condensates: statics and dynamics}.
		\bjtitle{Journal of Physics A: Mathematical and Theoretical}
		\bvolume{41}(\bissue{20}),
		\bfpage{205001}
		(\byear{2008})
	\end{barticle}
	\endbibitem
	
	\bibitem[\protect\citeauthoryear{Evans et~al.}{2006}]{evans2006interaction}
	\begin{barticle}
		\bauthor{\bsnm{Evans}, \binits{M.R.}},
		\bauthor{\bsnm{Hanney}, \binits{T.}},
		\bauthor{\bsnm{Majumdar}, \binits{S.N.}}:
		\batitle{Interaction-driven real-space condensation}.
		\bjtitle{Phys. Rev. Lett.}
		\bvolume{97},
		\bfpage{010602}
		(\byear{2006})
	\end{barticle}
	\endbibitem
	
	\bibitem[\protect\citeauthoryear{Evans and
		Waclaw}{2015}]{evans2015condensation}
	\begin{barticle}
		\bauthor{\bsnm{Evans}, \binits{M.R.}},
		\bauthor{\bsnm{Waclaw}, \binits{B.}}:
		\batitle{Condensation in models with factorized and pair-factorized stationary
			states}.
		\bjtitle{Journal of Statistical Mechanics: Theory and Experiment}
		\bvolume{2015}(\bissue{9}),
		\bfpage{09005}
		(\byear{2015})
	\end{barticle}
	\endbibitem
	
	\bibitem[\protect\citeauthoryear{Jatuviriyapornchai
		et~al.}{2020}]{jatuviriyapornchai2020structure}
	\begin{barticle}
		\bauthor{\bsnm{Jatuviriyapornchai}, \binits{W.}},
		\bauthor{\bsnm{Chleboun}, \binits{P.}},
		\bauthor{\bsnm{Grosskinsky}, \binits{S.}}:
		\batitle{Structure of the condensed phase in the inclusion process}.
		\bjtitle{Journal of Statistical Physics}
		\bvolume{178},
		\bfpage{682}--\blpage{710}
		(\byear{2020})
	\end{barticle}
	\endbibitem
	
	\bibitem[\protect\citeauthoryear{Feng}{2010}]{feng2010poisson}
	\begin{bbook}
		\bauthor{\bsnm{Feng}, \binits{S.}}:
		\bbtitle{The Poisson-Dirichlet Distribution and Related Topics: Models and
			Asymptotic Behaviors}.
		\bpublisher{Springer},
		\blocation{Berlin}
		(\byear{2010})
	\end{bbook}
	\endbibitem
	
	\bibitem[\protect\citeauthoryear{Chleboun et~al.}{2024}]{chleboun2023size}
	\begin{barticle}
		\bauthor{\bsnm{Chleboun}, \binits{P.}},
		\bauthor{\bsnm{Gabriel}, \binits{S.}},
		\bauthor{\bsnm{Grosskinsky}, \binits{S.}}:
		\batitle{Size-biased diffusion limits and the inclusion process}.
		\bjtitle{Electronic Journal of Probability}
		\bvolume{29},
		\bfpage{1}--\blpage{36}
		(\byear{2024})
	\end{barticle}
	\endbibitem
	
	\bibitem[\protect\citeauthoryear{Chleboun et~al.}{2022}]{chleboun2022poisson}
	\begin{barticle}
		\bauthor{\bsnm{Chleboun}, \binits{P.}},
		\bauthor{\bsnm{Gabriel}, \binits{S.}},
		\bauthor{\bsnm{Grosskinsky}, \binits{S.}}:
		\batitle{Poisson-dirichlet asymptotics in condensing particle systems}.
		\bjtitle{Electronic Journal of Probability}
		\bvolume{27},
		\bfpage{1}--\blpage{35}
		(\byear{2022})
	\end{barticle}
	\endbibitem
	
	\bibitem[\protect\citeauthoryear{Andjel}{1982}]{andjel1982invariant}
	\begin{barticle}
		\bauthor{\bsnm{Andjel}, \binits{E.D.}}:
		\batitle{Invariant measures for the zero range process}.
		\bjtitle{The Annals of Probability}
		\bvolume{10}(\bissue{3}),
		\bfpage{525}--\blpage{547}
		(\byear{1982})
	\end{barticle}
	\endbibitem
	
	\bibitem[\protect\citeauthoryear{Grosskinsky and
		Sch{\"u}tz}{2008}]{schuetz2008}
	\begin{barticle}
		\bauthor{\bsnm{Grosskinsky}, \binits{S.}},
		\bauthor{\bsnm{Sch{\"u}tz}, \binits{G.M.}}:
		\batitle{Discontinuous condensation transition and nonequivalence of ensembles
			in a zero-range process}.
		\bjtitle{Journal of statistical physics}
		\bvolume{132},
		\bfpage{77}--\blpage{108}
		(\byear{2008})
	\end{barticle}
	\endbibitem
	
	\bibitem[\protect\citeauthoryear{Chleboun and
		Grosskinsky}{2010}]{chleboun2010finite}
	\begin{barticle}
		\bauthor{\bsnm{Chleboun}, \binits{P.}},
		\bauthor{\bsnm{Grosskinsky}, \binits{S.}}:
		\batitle{Finite size effects and metastability in zero-range condensation}.
		\bjtitle{Journal of Statistical Physics}
		\bvolume{140}(\bissue{5}),
		\bfpage{846}--\blpage{872}
		(\byear{2010})
	\end{barticle}
	\endbibitem
	
	\bibitem[\protect\citeauthoryear{Grosskinsky}{2008}]{grosskinsky2008equivalence}
	\begin{barticle}
		\bauthor{\bsnm{Grosskinsky}, \binits{S.}}:
		\batitle{Equivalence of ensembles for two-species zero-range invariant
			measures}.
		\bjtitle{Stochastic Processes and their Applications}
		\bvolume{118}(\bissue{8}),
		\bfpage{1322}--\blpage{1350}
		(\byear{2008})
	\end{barticle}
	\endbibitem
	
	\bibitem[\protect\citeauthoryear{Chleboun and
		Grosskinsky}{2014}]{chleboun2014condensation}
	\begin{barticle}
		\bauthor{\bsnm{Chleboun}, \binits{P.}},
		\bauthor{\bsnm{Grosskinsky}, \binits{S.}}:
		\batitle{Condensation in stochastic particle systems with stationary product
			measures}.
		\bjtitle{Journal of Statistical Physics}
		\bvolume{154}(\bissue{1-2}),
		\bfpage{432}--\blpage{465}
		(\byear{2014})
	\end{barticle}
	\endbibitem
	
	\bibitem[\protect\citeauthoryear{Huveneers and
		Theil}{2019}]{huveneers2019equivalence}
	\begin{barticle}
		\bauthor{\bsnm{Huveneers}, \binits{F.}},
		\bauthor{\bsnm{Theil}, \binits{E.}}:
		\batitle{Equivalence of ensembles, condensation and glassy dynamics in the
			bose-hubbard hamiltonian}.
		\bjtitle{Journal of Statistical Physics}
		\bvolume{177}(\bissue{5}),
		\bfpage{917}--\blpage{935}
		(\byear{2019})
	\end{barticle}
	\endbibitem
	
	\bibitem[\protect\citeauthoryear{Davis and
		McDonald}{1995}]{davis1995elementary}
	\begin{barticle}
		\bauthor{\bsnm{Davis}, \binits{B.}},
		\bauthor{\bsnm{McDonald}, \binits{D.}}:
		\batitle{An elementary proof of the local central limit theorem}.
		\bjtitle{Journal of Theoretical Probability}
		\bvolume{8},
		\bfpage{693}--\blpage{701}
		(\byear{1995})
	\end{barticle}
	\endbibitem
	
	\bibitem[\protect\citeauthoryear{Ferrari and Sisko}{2007}]{ferrari2007escape}
	\begin{botherref}
		\oauthor{\bsnm{Ferrari}, \binits{P.A.}},
		\oauthor{\bsnm{Sisko}, \binits{V.V.}}:
		Escape of mass in zero-range processes with random rates.
		Lecture Notes-Monograph Series,
		108--120
		(2007)
	\end{botherref}
	\endbibitem
	
	\bibitem[\protect\citeauthoryear{Rafferty
		et~al.}{2018}]{rafferty2018monotonicity}
	\begin{bchapter}
		\bauthor{\bsnm{Rafferty}, \binits{T.}},
		\bauthor{\bsnm{Chleboun}, \binits{P.}},
		\bauthor{\bsnm{Grosskinsky}, \binits{S.}}:
		\bctitle{Monotonicity and condensation in homogeneous stochastic particle
			systems}.
		In: \bbtitle{Annales de l'Institut Henri Poincare (B) Probability and
			Statistics},
		vol. \bseriesno{54},
		pp. \bfpage{790}--\blpage{818}
		(\byear{2018})
	\end{bchapter}
	\endbibitem
	
	\bibitem[\protect\citeauthoryear{Kafri et~al.}{2002}]{kafri2002criterion}
	\begin{barticle}
		\bauthor{\bsnm{Kafri}, \binits{Y.}},
		\bauthor{\bsnm{Levine}, \binits{E.}},
		\bauthor{\bsnm{Mukamel}, \binits{D.}},
		\bauthor{\bsnm{Sch{\"u}tz}, \binits{G.M.}},
		\bauthor{\bsnm{T{\"o}r{\"o}k}, \binits{J.}}:
		\batitle{Criterion for phase separation in one-dimensional driven systems}.
		\bjtitle{Physical review letters}
		\bvolume{89}(\bissue{3}),
		\bfpage{035702}
		(\byear{2002})
	\end{barticle}
	\endbibitem
	
	\bibitem[\protect\citeauthoryear{Evans et~al.}{2006}]{evans2006canonical}
	\begin{barticle}
		\bauthor{\bsnm{Evans}, \binits{M.R.}},
		\bauthor{\bsnm{Majumdar}, \binits{S.N.}},
		\bauthor{\bsnm{Zia}, \binits{R.K.}}:
		\batitle{Canonical analysis of condensation in factorised steady states}.
		\bjtitle{Journal of Statistical Physics}
		\bvolume{123}(\bissue{2}),
		\bfpage{357}--\blpage{390}
		(\byear{2006})
	\end{barticle}
	\endbibitem
	
	\bibitem[\protect\citeauthoryear{Feng}{2021}]{feng2021note}
	\begin{barticle}
		\bauthor{\bsnm{Feng}, \binits{S.}}:
		\batitle{A note on residual allocation models}.
		\bjtitle{Frontiers of Mathematics in China}
		\bvolume{16}(\bissue{2}),
		\bfpage{381}--\blpage{394}
		(\byear{2021})
	\end{barticle}
	\endbibitem
	
\end{thebibliography}


\end{document}